\documentclass[11pt]{article}

\usepackage{amsmath, amssymb, amsthm} 
\usepackage{hyperref}

\allowdisplaybreaks

\theoremstyle{plain}
\newtheorem{theorem}{Theorem}[section]
\newtheorem{lemma}[theorem]{Lemma}

\theoremstyle{remark}
\newtheorem{remark}[theorem]{Remark}
\theoremstyle{definition}
\newtheorem{definition}[theorem]{Definition}

\newcommand{\whp}{w.h.p\@.}

\DeclareMathOperator{\minrank}{min-rank}
\RequirePackage[normalem]{ulem} 
\RequirePackage{color}\definecolor{RED}{rgb}{1,0,0}\definecolor{BLUE}{rgb}{0,0,1} 

\begin{document} 
\title{The Minrank of Random Graphs over Arbitrary Fields}
\author{Noga Alon 
\thanks {Department of Mathematics, Princeton University, Princeton,
	New Jersey, USA; and Schools of Mathematics and Computer Science,
	Tel Aviv University, Tel Aviv, Israel. Email: {\tt nogaa@tau.ac.il}.
	Research supported in part by ISF grant No. 281/17, GIF grant No. G-1347-304.6/2016 and the Simons Foundation.}
\and Igor Balla \thanks{Department of Mathematics, ETH, 8092 Zurich. Email: igor.balla@math.ethz.ch.}
\and Lior Gishboliner \thanks{School of Mathematical Sciences, Tel Aviv University, Tel Aviv, 69978, Israel. Email: {\tt liorgis1@post.tau.ac.il}. Research supported by ERC Starting Grant 633509.}
\and Adva Mond \thanks{School of Mathematical Sciences, Tel Aviv University, Tel Aviv, 69978, Israel. Email: {\tt advamond@post.tau.ac.il}.}
\and Frank Mousset\thanks{
School of Mathematical Sciences, Tel Aviv University, Tel Aviv, 
Israel. Email: {\tt moussetfrank@gmail.com}.
Research supported by ISF grants 1028/16 and 1147/14, and ERC Starting Grant
633509.}
}
\maketitle
\begin{abstract}
The minrank of a graph $G$ on the set of vertices $[n]$ over a field $\mathbb{F}$ is the minimum possible rank of a
  matrix $M\in\mathbb{F}^{n\times n}$ 
  with nonzero diagonal entries such that
  $M_{i,j}=0$ whenever $i$ and $j$ are distinct nonadjacent vertices of $G$.
   This notion, over the real field, arises in the study of 
the Lov\'asz theta function of a graph. We obtain tight bounds for
the typical minrank of the binomial random graph $G(n,p)$ over any
finite or infinite field, showing that for every field
  $\mathbb{F}=\mathbb F(n)$ 
  and every $p=p(n)$ satisfying $n^{-1} \leq p \leq
  1-n^{-0.99}$, the minrank of $G=G(n,p)$ over $\mathbb{F}$ is
$\Theta(\frac{n \log (1/p)}{\log n})$ with high probability. The result for the real field
settles a problem raised by Knuth in 1994. The proof combines a
recent argument of Golovnev, Regev, and Weinstein, who proved the 
above result for finite fields of size at most $n^{O(1)}$, with tools from
linear algebra, including an estimate
of R\'onyai, Babai, and Ganapathy for the number of zero-patterns of
a sequence of polynomials.
\end{abstract}

\section{Introduction}
In this paper we discuss the notion of the minrank of a graph over
a field, defined as follows.
\begin{definition}
\label{def:minrank} 
  The minrank of a graph $G$ on the vertex set $[n]=\{1,2,
  \dots ,n\}$ over a field $\mathbb{F}$, denoted by
  $\minrank_{\mathbb{F}}(G)$, is the minimum possible rank of a
  matrix $M\in \mathbb F^{n\times n}$ with nonzero diagonal entries
  such that $M_{i,j}=0$ whenever $i$ and $j$ are distinct nonadjacent vertices
  of $G$.
\end{definition}
The notion of the minrank over the real field $\mathbb R$ (with the added
requirement that the representing matrix $M$ be positive
semidefinite) arises in the study of
orthogonal representations of graphs, which play an important role in the
definition of the Lov\'asz theta function of a graph and its
relation to the study of the Shannon capacity, see \cite{Kn,Lo}.
The minrank over finite fields has been studied for its connections
to the Shannon capacity \cite{Ha} and
to linear index
coding \cite{BBJK}.
Knuth \cite{Kn} raised the problem of 
determining the asymptotic behavior of the typical minrank
of the binomial  random  graph $G(n,p)$ over the real field 
for fixed $p \in (0,1)$ and
large $n$, mentioning that it is at least $\Omega(\sqrt n)$. The analogous problem over finite fields was raised by Lubetzky and Stav \cite{LS} also in the context of linear index coding. Haviv and Langberg \cite{HL} proved a lower bound of $\Omega(\sqrt{n})$ for the minrank of $G(n,p)$ over any fixed finite field and for any constant $p$.
In a recent beautiful paper, Golovnev, Regev, and Weinstein 
\cite{GRW} substantially improved the aformentioned results by showing that for any finite field ${\mathbb{F}}=\mathbb F(n)$ and every
$p=p(n)$ in $(0,1)$, the minrank of the random graph $G(n,p)$ over
${\mathbb{F}}$ is
with high probability 
at least
$$
\Omega\left(\frac{n \log (1/p)}{\log (n|{\mathbb{F}}|/p)}\right).
$$
Since the minrank of every graph over any
field is at most the chromatic number of its complement, the known 
results about the behavior of the chromatic number of random
graphs show that the above estimate is tight up to a constant
factor for every finite
field of size  at most $n^{O(1)}$ and for every $p$ which is not
too close to $0$ or $1$, e.g., for all $n^{-1} \leq p \leq
1-n^{-0.99}$. This, however, provides no information
for infinite fields, and in particular for the real field.

Our main result is an extension of this
result to all finite or infinite fields. Here and in what
follows, the expression ``with high probability'' (\whp{} for
short)
means ``with probability tending to $1$ as $n$ goes to infinity''.
\begin{theorem}
\label{thm:min_rank} 
  Let $\mathbb F=\mathbb F(n)$ be a field and assume that $p = p(n)$ satisfies $n^{-1}\leq p \leq 1$.
  Then \whp{}
  \[ 
    \minrank_{\mathbb{F}}(G(n,p)) \geq \frac{n \log(1/p)}{80\log n}. \]
\end{theorem} 
The proof combines the method of 
Golovnev, Regev, and Weinstein
with tools from linear
algebra, most notably an estimate
of R\'onyai, Babai, and Ganapathy \cite{RBG}  for the number of zero patterns of
a sequence of multivariate polynomials over a field.

The result for the real field settles the problem of Knuth mentioned above. We
conclude this introduction by making several remarks. 

\subsection{Remarks}

\begin{remark}
  [Tightness of Theorem \ref{thm:min_rank}]
\label{remark:tight}
Theorem \ref{thm:min_rank} is tight up to the
value of the multiplicative constant $\frac{1}{80}$ for every field $\mathbb F$
and 
  every $n^{-1} \leq p \leq 1-n^{-0.99}$.
Indeed, for every graph $G$
and every field ${\mathbb{F}}$ we have $\minrank_{{\mathbb{F}}}(G) 
\leq \chi(\overline{G})$\footnote{Given a proper coloring of
$\overline{G}$, define $M$ by $M_{i,j} = 0$ if $i,j$ lie in different color
classes, and $M_{i,j} = 1$ otherwise. It is easy to see that the rank
of $M$ is the number of colors, and that $M_{i,j} = 0$ whenever $(i,j)
  \notin E(G)$.}, and it is known that for $p$ in the above range, $G = G(n,p)$ \whp{} satisfies $\chi(\overline{G}) = \Theta(\frac{n \log(1/p)}{\log n})$, 
see \cite{Bo,JLR}. This proves the following result:
\end{remark} 
\begin{theorem}
\label{t12} 
  Let $\mathbb F=\mathbb F(n)$ be a field and assume that $p = p(n)$ satisfies
  $n^{-1}\leq p \leq 1- n^{-0.99}$. Then \whp{}
  \[ 
    \minrank_{\mathbb{F}}(G(n,p)) =\Theta\left(\frac{n \log(1/p)}{\log n}\right). \]
\end{theorem} 

\begin{remark} 
  [Amplifying the success probability in Theorem \ref{thm:min_rank}]
  \label{remark:concentration}  
  The proof of
  Theorem \ref{thm:min_rank} gives a bound of $n^{-\Omega(1)}$ on the
  probability that $G = G(n,p)$ (for $p \geq n^{-1}$)
  satisfies $\minrank_{\mathbb{F}}(G)
  < \frac{n \log(1/p)}{80\log n}$. Using Azuma's inequality for the
  vertex exposure martingale, one can show that $\minrank_{\mathbb{F}}(G)$
  is highly concentrated around its expectation, which is $\Omega(
  \frac{n \log(1/p)}{\log n} )$ by Theorem \ref{thm:min_rank}. This
  way, one can deduce that $\minrank_{\mathbb{F}}(G) \geq \Omega(
  \frac{n \log(1/p)}{\log n} )$ holds with probability at least $1 -
  e^{-\Omega(n/\log^2 n)}$, see~\cite{GRW} for a detailed argument.
\end{remark} 

\paragraph{Paper organization}
The proof of
Theorem~\ref{thm:min_rank} is given in Sections 2 and 3. Section 4 contains a
number of applications of our main theorem to the study of various geometric
representations of random graphs. Section 5 contains some concluding remarks
and open problems.

\section{Preliminaries} 
\begin{definition}
  \label{def:sparse_matrix}
  An \emph{$(n,k,s)$-matrix} (over some field $\mathbb F$) is an $n \times n$ matrix $M$ of rank $k$
  with $s$ nonzero entries and containing rows $R_{i_1},\dots,R_{i_k}$
  and columns $C_{j_1},\dots,C_{j_k}$ such that $R_{i_1},\dots,R_{i_k}$
  is a row basis for $M$, $C_{j_1},\dots,C_{j_k}$ is a column basis for
  $M$, and the overall number of nonzero entries in all $2k$ vectors
  $R_{i_1},\dots,R_{i_k},C_{j_1},\dots,C_{j_k}$ is at most $4ks/n$.
\end{definition}
\noindent
The following is the key lemma in \cite{GRW}.
\begin{lemma}\label{lem:submatrix}
  Let $\mathbb F$ be any field and let $M\in \mathbb F^{n\times n}$ be a matrix of rank $k$.
  Then there exist integers $n'$, $k'$, and $s'$ with $k'/n'\leq k/n$ such that $M$ contains 
  an $(n',k',s')$-principal-submatrix (that is, a principal submatrix that is
  an $(n',k',s')$-matrix).
\end{lemma}

The {\em zero-pattern} of a sequence $(y_1,\dots,y_m) \in \mathbb{F}^m$ is the sequence $(z_1,\dots,z_m) \in \{0,\ast\}^m$ such that $z_i = 0$ if $y_i = 0$ and $z_i = \ast$ if $y_i \neq 0$. For a sequence of polynomials $\bar{f} = (f_1,\dots,f_m)$ over a field $\mathbb{F}$ in variables $X_1,\dots,X_N$, the {\em set of
zero-patterns} of $\bar{f}$ is the set of all zero-patterns of sequences obtained by assigning values from $\mathbb{F}$ to the variables $X_1,\dots,X_N$ in $(f_1,\dots,f_m)$. 
We define the zero-pattern of a matrix $M \in \mathbb{F}^{n \times n}$, and the set of zero-patters of a matrix whose entries are polynomials, by treating the matrix as a sequence of length $n^2$. 
R\'{o}nyai,
Babai, and Ganapathy \cite{RBG} gave the following bound for the number
of zero-patterns of a sequence of polynomials:
\begin{lemma}\label{lem:zero_patterns}
  Let $\bar{f} = (f_1,\dots,f_m)$ be a sequence of polynomials in $N$
  variables over a field $\mathbb{F}$, each of degree at most $d$. Then
  the number of zero-patterns of $\bar{f}$ is at most $\binom{md+N}{N}$.
\end{lemma}  
\noindent
We now state and prove the key lemma of this paper.
\begin{lemma}
  \label{lem:main}
  The number of zero-patterns of $(n,k,s)$-matrices is 
  at most $\binom{n}{k}^2 \cdot n^{20ks/n}$.
\end{lemma}
\begin{proof}
  It is easy to see that the lemma holds for $k \geq n-1$ (since the number of zero-patterns of $n \times n$ matrices with $s$ nonzero entries is clearly at most $\binom{n^2}{s} \leq n^{2s}$), so we may assume for
  convenience that $k \leq n-2$. The term $\binom{n}{k}^2$ corresponds to the
  number of ways to
  choose the sequences $(i_1,\dots,i_k)$ and $(j_1,\dots,j_k)$ from
  Definition \ref{def:sparse_matrix} (that is, the number of ways to
  choose the positions of the rows $R_{i_1},\dots,R_{i_k}$ and the
  columns $C_{j_1},\dots,C_{j_k}$). From now on we assume without loss
  of generality that $(i_1,\dots,i_k)=(j_1,\dots,j_k)=(1,\dots,k)$. The
  number of ways to choose a set $\mathcal{F} \subseteq ([k] \times [n])
  \cup ([n] \times [k])$ of at most $4ks/n$ entries which are allowed to be
  nonzero is at most
  $$
  \sum_{t = 0}^{4ks/n}{\binom{2kn}{t}} \leq 
  \left( \frac{e \cdot n^4}{s^2} \right)^{2ks/n} \leq n^{8ks/n}
  $$ 
  (the first inequality follows from the fact that for all $0<x<1$, we have
  $\sum_{t=0}^{4ks/n}\binom{2kn}{t}x^{4ks/n} \leq
  \sum_{t=0}^{4ks/n}\binom{2kn}{t}x^t\leq
  (1+x)^{2kn}
  \leq e^{x2kn}$ and setting $x=s/n^2<1$).
  So it is enough to show that for every fixed 
  $\mathcal{F} \subseteq ([k] \times [n]) \cup ([n] \times [k])$ as above,
  there are at most $n^{12ks/n}$ zero-patterns of matrices for which
  the first $k$ rows form a row basis, the first $k$ columns form a column
  basis, and for every $(i,j)$ with $\min(i,j) \leq k$ and $(i,j) \notin
  \mathcal{F}$, the $(i,j)$-entry is zero. Let $M = (M_{i,j})$
  be such a matrix, and denote by $M'$ the submatrix of $M$ on $[k] \times
  [k]$. 
  We claim that $M'$ is invertible. Indeed, let $M''$ be the submatrix consisting of the first $k$ rows of $M$. Then $\text{rank}(M'') = k$ (because the rows of $M''$ form a row basis of $M$) and the columns of $M'$ span the column space of $M''$ (because the first $k$ columns of $M$ span its column space). It follows that the columns of $M'$ are linearly independent, as required.   
 
  Fix any $k+1 \leq \ell \leq
  n$. The $\ell$-th column of $M$ is a linear combination of the first
  $k$ columns of $M$. The coefficients in this linear combination are the
  coordinates of the unique solution to the system
  $$
  M' \cdot x = 
  \begin{pmatrix}
    M_{1,\ell} \\
    \vdots \\
    M_{k,\ell}
  \end{pmatrix}.
  $$  
  By Cramer's rule, this solution can be expressed as 
  $$		
  \begin{pmatrix}
    f_{1,\ell}(y^{(\ell)})/\det(M') \\
    \vdots \\
    f_{k,\ell}(y^{(\ell)})/\det(M')
  \end{pmatrix},
  $$
  where $f_{1,\ell},\dots,f_{k,\ell}$ are polynomials of degree $k$
  (which do not depend on the matrix $M$), and the vector of variables
  $y^{(\ell)}$ contains the entries $(M_{i,j})_{i,j=1}^{k}$ and
  $(M_{i,\ell})_{i=1}^{k}$. We see that for every $k+1 \leq \ell \leq n$
  and $1 \leq i \leq n$, we have
  $$
  M_{i,\ell} = \frac{1}{\det(M')}\sum_{j=1}^{k}{f_{j,\ell}(y^{(\ell)}) 
  \cdot M_{i,j}}.
  $$
  This means that every entry of $M$ can be given as a polynomial of degree
  $k+1$ in the entries $(M_{i,j})_{\min(i,j) \leq k}$, divided by
  the nonzero polynomial $\det(M')$. Since $M_{i,j} = 0$
  if $\min(i,j) \leq k$ and $(i,j) \notin \mathcal{F}$, it is enough to
  take $(M_{i,j})_{(i,j) \in \mathcal{F}}$ as the sequence of variables
  of all polynomials. We conclude that the zero-pattern of $M$ is the
  zero-pattern of a sequence of $n^2$ polynomials in $|\mathcal{F}|
  \leq 4ks/n$ variables, each of degree (at most) $k+1$
  (note that removing the factor $\frac{1}{\det(M')}$ does not change the
  zero-pattern of $M$, and that all polynomials are independent of $M$). By
  Lemma \ref{lem:zero_patterns}, the number of zero-patterns of this matrix
  of polynomials is at most
  $$
  \binom{(k+1)n^2 + 4ks/n}{4ks/n} \leq
  \binom{(k+2)n^2}{4ks/n} \leq 
	(k+2)^{4ks/n} n^{8ks/n} \leq 
  n^{12ks/n},
  $$
  as required.
\end{proof}
\noindent
Finally, we will need the following simple lemma from \cite{GRW}
(which follows, with a slightly better constant, from Tur\'an's Theorem). 
\begin{lemma}\label{lem:many_non-zero_entries}
  Every $n \times n$ matrix of rank $k$ having nonzero entries on the 
  main diagonal contains at least $n^2/(4k)$ nonzero entries. 
\end{lemma}

\section{The Min-rank of Random Graphs}
\begin{proof}[Proof of Theorem \ref{thm:min_rank}]
  If, say, $p \geq 1 - n^{-1}$, then $\frac{n \log(1/p)}{\log n} = o(1)$, so
  the theorem holds trivially. So from now on we assume that $p < 1 -
  n^{-1}$.

  Suppose that $G$ is an $n$-vertex graph with $\minrank_{\mathbb{F}}(G)
  \leq k$. Then by definition, there is a $n \times n$ matrix $M$ over
  $\mathbb{F}$ of rank at most $k$, such that all entries of $M$ on the main
  diagonal are nonzero, and such that $M_{i,j}
  = 0$
  whenever $(i,j) \notin E(G)$. By Lemma \ref{lem:submatrix}, $M$ contains an $(n',k',s')$-principal
  submatrix $M'$ with $k'/n' \leq k/n$.

  We conclude that for every graph $G$ satisfying $\minrank_{\mathbb{F}}(G)
  \leq k$, there is a set $U \subseteq V(G)$ and an $(n',k',s')$-matrix $M'$,
  where $n' = |U|$ and $k'/n' \leq k/n$, such that all entries
  of $M'$ on the main diagonal are nonzero, and such that for every pair of
  distinct $i,j \in U$, we have $M'_{i,j} = 0$ whenever  $(i,j) \notin E(G)$. For given
  $n',k',s'$, the number of choices for $U$ is $\binom{n}{n'}$, and the number
  of zero-patterns of $(n',k',s')$-matrices is at most $\binom{n'}{k'}^2 \cdot {n'}^{20k's'/n'}$
  by Lemma \ref{lem:main}. Fixing $U$ and the zero-pattern of $M'$, the probability that $G
  = G(n,p)$ satisfies the above event with respect to $U,M'$ is at most
  $p^{(s'-n')/2}$, since there are at least $(s' - n')/2$ pairs $1 \leq i
  < j \leq n'$ for which either $M'_{i,j} \neq 0$ or $M'_{j,i} \neq 0$, and each such pair must span an
  edge in $G$.  By Lemma \ref{lem:many_non-zero_entries} we have $s' \geq
  n'^2/(4k')\geq n' n/(4k)$. Hence, the probability that $G = G(n,p)$ satisfies
  $\minrank_{\mathbb{F}}(G) \leq k$ is at most
  \begin{align}\label{eq:union_bound} 
    &\nonumber\sum_{n' =1}^n \; 
      \sum_{k' =1}^{n'k/n} \sum_{s'\geq n'\cdot \frac{n}{4k}} {
      \binom{n}{n'} \cdot \binom{n'}{k'}^2 \cdot {n'}^{20k's'/n'}
      \cdot p^{(s' - n')/2}
    } \\ \leq{} &
      \sum_{n' =1}^n \; \sum_{k'=1}^{n'k/n} {
      n^{n'+2k'} \cdot p^{-n'/2}
      \sum_{s'\geq n'\cdot \frac{n}{4k}} {\left(n^{20k/n}p^{1/2}\right)^{s'}
      } }
  \end{align} 
  For $k \leq \frac{n
  \log(1/p)}{80\log n}$ we get $n^{20k/n} \leq (1/p)^{1/4}$, and so
  $$
  \sum_{s' \geq n' \cdot \frac{n}{4k}}{\left(n^{20k/n}p^{1/2}\right)^{s'} }
  \leq \sum_{s' \geq n' \cdot \frac{n}{4k}}{p^{s'/4}} \leq 
  p^{n' \cdot \frac{n}{16k}} \cdot \frac{1}{1 - p^{1/4}} \leq e^{-5n'\log n} \cdot
  \frac{1}{1 - p^{1/4}} = n^{-5n'} \cdot \frac{1}{1 - p^{1/4}}.  
  $$ 
  Hence, \eqref{eq:union_bound}
  is at most 
  \begin{align*} \frac{1}{1 - p^{1/4}} 
    \cdot \sum_{n' =1}^{n} \; \sum_{k'=1}
    ^{n'k/n} {n^{n'+2k'} \cdot p^{-n'} \cdot n^{-5n'}} \leq \frac{1}{1 - p^{1/4}} 
    \cdot \sum_{n'=1}^n {n^{4n'} \cdot n^{-5n'}} = \frac{1}{1 - p^{1/4}} 
    \cdot \sum_{n'=1}^n {n^{-n'}}.
  \end{align*}
  If (say) $p \leq 1/2$, then the above sum is clearly $o(1)$.
  In the complementary case $p > 1/2$ we have
  $k = O(n/\log n)$, and so in \eqref{eq:union_bound} we can restrict ourselves to
  $n'$ satisfying 
  $n' \geq n/k = \Omega(\log n)$ (as otherwise there
  are no $k'$ between $1$ and $n'k/n$). Now, recalling
  the assumption $p < 1 - n^{-1}$, we see that \eqref{eq:union_bound} evaluates to
  $o(1)$. This completes the proof.
\end{proof}

\section{Geometric Representations of Random Graphs}

\subsection{Orthogonal representations}

The parameter
$\minrank_{\mathbb{R}}(\cdot)$ is closely 
related to orthogonal representations
of graphs. An orthogonal representation of dimension $d$ of a graph $G$
is an assignment of nonzero vectors in $\mathbb{R}^d$ to the vertices
of $G$, so that the vectors corresponding to any nonadjacent pair are
orthogonal. Orthogonal representations of graphs were introduced by
Lov\'{a}sz in his seminal paper on the theta function \cite{Lo}
(see also \cite{Kn} for a survey on the subject). Note that if
$v_1,\dots,v_n \in \mathbb{R}^d$ is an orthogonal representation
of a graph $G$, then the Gram matrix $M$ of $v_1,\dots,v_n$ satisfies $M_{i,i} = \langle v_i,v_i \rangle \neq 0$ for each $1 \leq i \leq n$, and $M_{i,j} = \langle v_i,v_j \rangle = 0$
whenever $i,j$ are nonadjacent. Moreover, $M$ has rank at most $d$. 
It follows that the minimal $d$ for which $G$ has an orthogonal representation of dimension $d$ is at least as large as $\minrank_{\mathbb{R}}(G)$.\footnote{In fact, the minimal dimension of an orthogonal representation of $G$ is obtained by adding to the definition of the minrank (i.e. Definition \ref{def:minrank}) the restriction that $M$ is symmetric and positive semidefinite.}
Hence,
Theorem \ref{thm:min_rank} shows that the minimal dimension of an
orthogonal representation of $G = G(n,p)$ is \whp{}
$\Omega(\frac{n \log(1/p)}{\log n} )$. The same argument as in Remark
\ref{remark:tight} shows that this is tight for all 
$n^{-1} \leq p \leq 1 - n^{-0.99}$. This proves the following theorem, which settles a
problem of Knuth \cite{Kn}.
\begin{theorem}
\label{t41}
  For every $p=p(n)$ satisfying $n^{-1} \leq p \leq 1-n^{-0.99}$, the minimum
  dimension $d$ such that the random graph $G = G(n,p)$ has an orthogonal
  representation in $\mathbb{R}^d$ is, \whp{}, $\Theta(\frac{n
  \log(1/p)}{\log n})$.
\end{theorem}

\subsection{Unit distance graphs}

A {\em complete} unit-distance graph in $\mathbb{R}^d$ is a graph whose set
of vertices is a finite subset of the $d$-dimensional Euclidean space, where
two vertices are adjacent if and only if the Euclidean distance between them
is exactly $1$. A unit distance graph in $\mathbb{R}^d$ is any subgraph of
such a graph. Unit distance graphs have been considered in several papers,
see, e.g., \cite{AK} and the references therein. Note that if $u,v \in
\mathbb{R}^d$ are two adjacent vertices of such a unit distance graph, then
$\|u-v\|_2^2=1$. Let $u_1,u_2, \dots, u_n \in \mathbb{R}^d$ be  the vertices
of a unit distance graph $G$ in $\mathbb{R}^d$. Then the $n \times n$ matrix
$M$ defined by $M_{i,j}=1-\|u_i-u_j\|_2^2$ is a real matrix in
which every entry on the diagonal is $1$ and for every pair of distinct
adjacent vertices $u_i,u_j$, $M_{i,j}=0$. This implies that the rank of the
matrix $M$ must be at least $\minrank_{\mathbb R}(\overline{G}).$ On the other hand,
it is easy to see that the rank of $M$ is at most $d+2$. Indeed $M$ can be
expressed as a sum of three matrices $A,B,C$ where $A_{i,j}=1-\|u_i\|^2$,
$B_{i,j}=-\|u_j\|^2$ and $C_{i,j}=2u_i^t v_j$. As all columns of $A$ and all
rows of $B$ are identical, $A$ and $B$ are of rank $1$. The matrix $C$ is
twice the Gram matrix of vectors in $\mathbb{R}^d$, and hence its rank is at
most $d$. Therefore $M$ has rank at most $d+2$. It is also clear that every
graph of chromatic number $d$ is a unit distance graph in $\mathbb{R}^{d-1}$.
Indeed, the $d$ vertices $x_1, \dots ,x_d$ of a regular simplex of diameter
$1$ in $\mathbb{R}^{d-1}$ can be used to represent all vertices of $G$,
assigning  $x_i$ to all vertices in color class number $i$ of $G$, for $1
\leq i \leq d$. This establishes the following result.
\begin{theorem}
\label{t42}
  For every $p=p(n)$ satisfying
  $n^{-0.99} \leq p \leq 1-n^{-1}$, the minimum
  dimension $d$  such that the random graph
  $G = G(n,p)$ is a unit distance graph in $\mathbb{R}^d$ is,
  \whp{}, $\Theta(\frac{n \log(1/(1-p))}{\log n})$.
\end{theorem}

\subsection{Graphs of touching spheres}

The notion of a unit distance graph can be extended 
as follows. Call a graph $G$ on $n$ vertices a
graph of touching spheres in $\mathbb{R}^d$ if there are
spheres $S_1, S_2 \dots ,S_n$ in $\mathbb{R}^d$, where the sphere 
$S_i$ is centered at $u_i$ and its radius is $r_i$, and for every
pair of adjacent vertices $i$ and $j$, the two corresponding 
spheres $S_i$ and $S_j$ touch each other
and their convex hulls have disjoint interiors. That is, the
distance between $u_i$ and $u_j$ is exactly $r_i+r_j$. Note that if
$r_i=1/2$ for all $i$, then this is exactly the definition of a unit
distance graph. For $u_i$ and $r_i$ as above, the matrix
$M=(M_{i,j})$ where $M_{i,j}=(r_i+r_j)^2 -\|u_i-u_j\|^2$ 
has nonzero diagonal elements and $M_{i,j}=0$ for every pair of
adjacent vertices $i,j$. Furthermore, $M$ can be written as a sum of 
the four matrices in which the $(i,j)$-th entry is 
$r_i^2-\|u_i\|^2$, $r_j^2-\|u_j\|^2$, $r_ir_j$,
and $2 u_i^t u_j$, respectively. These have ranks at most $1,1,1$,
and $d$, respectively, showing that the rank of any such matrix $M$ is 
at most $d+3$. The  chromatic number of $G$ provides a
representation as before (even as a unit distance graph), implying
the following extension of Theorem \ref{t42}.
\begin{theorem}
\label{t43}
  For every $p=p(n)$ satisfying
  $n^{-0.99} \leq p \leq 1-n^{-1}$, the minimum
dimension $d$ such that the random graph
$G = G(n,p)$ is a graph of touching spheres in $\mathbb{R}^d$ is,
  \whp{}, $\Theta(\frac{n \log(1/(1-p))}{\log n})$.
\end{theorem}

\subsection{Graphs defined by a polynomial}

Let $P=P(x,y)=P(x_1,x_2, \dots ,x_d,y_1,y_2, \dots ,y_d)$,
where $x=(x_1, \dots x_d)$ and $y=(y_1, \dots ,y_d)$, be a
polynomial of $2d$ variables over a field $\mathbb{F}$, and assume
that it satisfies
$P(x,y)=P(y,x)$
for all $x,y \in \mathbb{F}^d$.
Say that a graph $G$ on $n$ vertices $1,2, \dots ,n$ is
a  $P$-graph over $\mathbb{F}^d$ if there are
vectors $x^{(1)}, \dots ,x^{(n)} \in \mathbb{F}^d$
such that $P(x^{(i)},x^{(i)}) \neq 0$ for all $1 \leq i \leq
n$, and for every pair of distinct adjacent vertices
$i,j$, $P(x^{(i)},x^{(j)})=0$. Thus, for example, unit distance
graphs correspond to the polynomial $1-\|x-y\|^2$. We will often think of $P$
as a sequence of polynomials, indexed by $n$ (so the number of variables is
allowed to grow with $n$).

For any $P$-graph as above, the matrix $M$ given by
$M_{i,j}=P(x^{(i)},x^{(j)})$ vanishes in every entry corresponding to
adjacent vertices, and has nonzero entries on the main diagonal. If the
degree of $P$ is large, then even a small number of variables $2d$ can be
enough to represent all $n$-vertex graphs as $P$-graphs. Indeed, if for
example, the field is $\mathbb{F}_2$, $d=\log_2 n$, and $P=\prod_{i=1}^d
(1+x_i+y_i)$, then for the set $X=\{0,1\}^d$ of $n=2^d$ vertices, we have
$P(x,x') \neq 0$ if and only if $x = x'$, meaning that 
every graph on  $n$ vertices is a $P$-graph, although the number of
variables is only $O(\log n)$. On the other hand, if $P$ is of
degree at most $3$, it is not difficult to see that if $G$ is a
$P$-graph then the rank of the
matrix $M$ defined as above is at most $2d+1$. To see this, write
$P$ in the form 
$$
P=c+\sum_{i=1}^d x_i f_i(y)+\sum_{j=1}^d y_j h_j(x).
$$
Next define, for each vector $x=(x_1,x_2, \dots ,x_d)$, two vectors $F(x)$
and $H(x)$ of length $2d+1$ each, as follows:
$$
F(x)=(1,x_1,x_2, \dots ,x_d, h_1(x),h_2(x), \dots ,h_d(x)),
$$
$$
H(x)=(c,f_1(x),f_2(x), \dots ,f_d(x), x_1, x_2, \dots ,x_d).
$$
Thus for every $x=(x_1, \dots ,x_d)$ and $y=(y_1, \dots, y_d)$, 
$P(x,y)$ is exactly the inner product of $F(x)$ with $H(y)$. This
shows that if $G$ is a $P$-graph then the matrix $M$ above can be
written as a product of the $n \times (2d+1)$ matrix whose rows are
the vectors $F(x)$ by the $(2d+1) \times n$ matrix whose columns are the vectors $H(x)$ (where in both cases $x$ goes over all vectors representing the vertices of $G$). This shows that indeed the rank of $M$ is at most $2d+1$
and implies the following.
\begin{theorem}\label{thm:P_graphs}
  \label{p44}
  Let $P=P(x_1, x_2, \dots ,x_d,y_1, y_2, \dots ,y_d)$ be a polynomial of
  degree at most $3$ over a field $\mathbb{F}$ and let $p = p(n)$ satisfy $
  n^{-1}\leq p\leq 1$.
  If the random graph $G = G(n,p)$ is a $P$-graph with probability
  $\Omega(1)$, then $d$ is at least
  $\Omega(\frac{n \log(1/(1-p))}{\log n})$.
\end{theorem}
\noindent
Note that the proof above works for every polynomial $P$ with
$O(d)$ monomials like, for example,
$$
P(x,y)= 1- \|x-y\|_4^4=1-\sum_{i=1}^d (x_i-y_i)^4,
$$ 
or even every polynomial $P$ which is the sum
of $O(d)$ terms, each being either a product of a monomial in the
variables of $x$ times any function of those in $y$, or vice versa.

\subsection{Spaces of polynomials}

For a field $\mathbb{F}$ and a linear space
$\mathbb{S}$
of polynomials 
in $\mathbb{F}[x_1, x_2 \dots ,x_m]$, a graph $G$ on
the vertices $1, 2, \dots ,n$ has a representation over
$\mathbb{S}$ if for every vertex $i$ there are $P_i \in \mathbb{S}$
and $v_i \in \mathbb{F}^m$ so that $P_i(v_i) \neq 0$ for all $i$
and 
$P_i(v_j)=0$
for every two distinct nonadjacent vertices $i$ and $j$.
As shown in~\cite{Al1}, if $G$ has such a
representation then its Shannon capacity is at most the dimension
of the linear space $\mathbb{S}$. It is easy to see that the
rank of the matrix $M=(M_{i,j})=(P_i(v_j))$ is at most
the dimension of $\mathbb{S}$. Therefore we get the following.
\begin{theorem}
  \label{p45}
  Let $\mathbb{S}$ be a linear space of polynomials 
  in variables $x_1, x_2, \dots ,x_m$ over a field 
  $\mathbb{F}$ and let $p=p(n)$ satisfy $n^{-1}\leq p \leq 1$.
  If
  $G = G(n,p)$ has a representation over $\mathbb{S}$
  with probability $\Omega(1)$,
  then 
  $\dim \mathbb{S}$ is at least $\Omega(\frac{n \log(1/p)}{\log n})$.
\end{theorem}

\section{Concluding Remarks and Open Problems}
\begin{itemize}
\item
  We have shown that for all $n^{-1} \leq p \leq
1-n^{-0.99}$ and for any finite or infinite field
$\mathbb{F}$, the minrank of the random  graph $G(n,p)$ over
$\mathbb{F}$ satisfies, \whp, 
$\minrank_{\mathbb{F}}(G)=\Theta(\frac{n \log(1/p)}{\log n})$.
    For $p =n^{-1}$  this gives a lower bound of $\Omega(n)$, and
as $n$ is always a trivial upper bound and the function
$\minrank_{\mathbb{F}}(G)$ for $G=G(n,p)$ is clearly monotone
decreasing in $p$, it follows that for all $0 \leq p \leq
    n^{-1}$, $\minrank_{\mathbb{F}}(G)=\Theta(n)$. 
    In the other extreme, for $p \geq 1 - O(n^{-1})$,
    the graph $G = G(n,p)$ satisfies \whp{} $\chi(\overline{G}) = \Theta(1)$, 
and hence $\minrank_{{\mathbb{F}}}(G) = \Theta(1)$. So the only 
regime in which there is a gap of more than a constant factor
 between the lower bound of 
Theorem \ref{thm:min_rank} and the typical value of $\chi(G(n,1-p))$
    is when $\omega(n^{-1}) \leq 1-p \leq n^{-1+o(1)}$. 

In all the range of $p$ discussed in this paper, 
the minrank of a graph is equal, up to a
constant factor, to the chromatic number of its complement. 
It will be interesting to decide how close these two quantities really are,
and in particular, to decide whether or not for $G=G(n,1/2)$,
$$
\minrank_{\mathbb{R}}(G) = (1+o(1)) \chi(\overline{G})   
~~\left( ~=(1+o(1)) \frac{n}{2 \log_2 n} \right).
$$
\item
It was shown in \cite{Ha} that the
minrank of a graph over any field is an upper bound for its
Shannon capacity. In particular, the infimum of
$\minrank_{\mathbb{F}}(G)$
over all fields $\mathbb F$ is such an upper bound.
Combining our technique here with a recent result of
Nelson (Theorem 2.1 in \cite{Ne})
that extends the one of \cite{RBG}, we can show
that for the random graph $G(n,1/2)$ this bound is
weaker than the theta function, which is $\Theta(\sqrt n)$~\cite{Ju}. More generally, we have the following.
\begin{theorem}\label{thm:all_fields}
\label{t51}
For every $n^{-1} \leq p \leq 1$, the random graph $G=G(n,p)$ satisfies \whp{} that
$\minrank_{{\mathbb{F}}}(G) \geq \Omega\left( \frac{n\log(1/p)}{\log n} \right)$ for every field
$\mathbb{F}$.
\end{theorem}
It is worth noting that it follows from results of Grosu \cite{Grosu} and of
Tao \cite{Tao} that the minrank of a graph $G$ over $\mathbb C$
is a lower bound for its minrank over every field $\mathbb F$ whose
characteristic is  sufficiently large
as a function of $G$.
By combining these results with Theorem \ref{thm:min_rank}, we immediately get that for every $n^{-1} \leq p \leq 1$, the random graph $G = G(n,p)$ \whp{} satisfies $\minrank_{{\mathbb{F}}}(G) \geq \Omega\left( \frac{n\log(1/p)}{\log n} \right)$ for every field $\mathbb{F}$ of characteristic
which is sufficiently large as a function of $n$.
The stronger assertion of Theorem \ref{thm:all_fields} follows by
replacing the result of \cite{RBG} (Lemma~\ref{lem:zero_patterns}) by that of \cite{Ne} in the
proof of Theorem  \ref{thm:min_rank}.
%
\item
  In general, the minrank of a graph may depend heavily
  on the choice of the field. 
  To see this we use the well-known fact that for any
  graph $G$ on $n$ vertices and for any 
  field ${\mathbb{F}}$, 
  $$\minrank_{\mathbb{F}}(G) \cdot 
  \minrank_{\mathbb{F}}(\overline{G}) \geq n.$$ Indeed, if $A=(A_{i,j})$
  and $B=(B_{i,j})$ are representations  for $G$ and its complement over
  ${\mathbb{F}}$ (as in Definition \ref{def:minrank}), 
  then the matrix $(A_{i,j} \cdot B_{i,j})$ has nonzeros on the 
  diagonal and zero in every other entry, hence its rank is $n$. 
  As it is a submatrix of the tensor product of $A$ and
  $B$, its rank is at most the product of their ranks, proving the
  above inequality. On the other hand, \cite{Al1} contains an
  example of a family of graphs $G_n$ on $n$ vertices satisfying
  $\minrank_{\mathbb{F}_p}(G_n),
  \minrank_{\mathbb{F}_q}(\overline{G_n})
  \leq n^{o(1)}$, where $\mathbb{F}_p$ and $\mathbb{F}_q$ are 
  two distinct appropriately chosen prime fields (with $p$ and $q$
    depending on $n$).
  An even more substantial gap between the minrank of a graph
  over a finite field and its minrank over the reals, 
  at least when insisting on a representation by a 
  positive semi-definite matrix,
  is given in \cite{Al2}, which provides an example of a sequence of graphs $G_n$ on $n$
  vertices for which $\minrank_{\mathbb{F}}(G_n) =3$ for some finite
  field ${\mathbb{F}}$ (depending on $n$), 
  whereas 
  the minimum possible rank over the reals by a positive
  semi-definite matrix is greater than $n^{1/4}$.

\item
  Haviv has recently combined the key lemma of \cite{GRW}
  with the Lov\'asz Local Lemma and proved a related result.
  To state it, we use the following density parameter.
  For a graph $H$ with $h \geq 3$ vertices, let $m_2(H)$ denote the maximum value of 
  $\frac{f'-1}{h'-2}$ over all pairs $(h',f')$ such that there is a
  subgraph $H'$ of $H$ with $h' \geq 3$ vertices and $f'$ edges.
  \noindent
  \begin{theorem}[Haviv \cite{Hav}]
    \label{t51}
    Let $H$ be a graph with $h \geq 3$ vertices and $f $ edges.
    Then there is some
    $c=c(H)>0$ so that for every finite field $\mathbb{F}$ 
    and every integer $n$
    there is a graph $G$ on $n$ vertices whose complement
    contains no copy of $H$, so that 
    $$
    \minrank_{\mathbb{F}} (G) \geq c \frac{n^{1-1/m_2(H)}}
    {\log (n |\mathbb{F}|)}.
    $$
  \end{theorem}
  Combining the proof of Haviv with our approach here 
  we can get rid of the dependence on the size of the field and prove
  the following stronger result.
  \begin{theorem}
    \label{t52}
    Let $H$ be a graph with $h \geq 3$ vertices.
    Then there is a constant $c=c(H)>0$ such that
    for every finite or infinite field
    $\mathbb{F}$ and every integer $n$ 
    there is a graph $G$ on $n$ vertices whose complement
    contains no copy of $H$, so that 
    $$
    \minrank_{\mathbb{F}}(G) \geq c \frac{n^{1-1/m_2(H)}}
    {\log n}.
    $$
  \end{theorem}
  We omit the detailed proof.
\item
  The results in this paper are formulated for undirected graphs, but
  can be easily extended  to the directed case, with essentially the
  same proofs. In particular, Theorem \ref{thm:P_graphs} also holds for digraphs defined by a polynomial (in this case we do not need to assume that the polynomial $P$ is symmetric,  see Subsection 4.4). 
\end{itemize}

\paragraph{Acknowledgements}
We thank Peter Nelson for telling us about his paper \cite{Ne}.
The research on this project was initiated during a joint research
workshop of Tel Aviv University and the Free University of Berlin
on
Graph and Hypergraph Coloring Problems, held in Berlin in August
2018,
and supported by a GIF grant number G-1347-304.6/2016. We would
like to
thank the German-Israeli Foundation (GIF) and
both institutions for their support.

\end{document}